\documentclass[11pt]{amsart}
\usepackage{stmaryrd}
\usepackage{amsmath}
\usepackage{cases}
\usepackage{mathrsfs}
\usepackage{amssymb}
\usepackage{amsfonts}
\usepackage{amssymb,amsfonts}

\allowdisplaybreaks

\begin{document}
\theoremstyle{plain}
\newtheorem{thm}{Theorem}[section]
\newtheorem{theorem}[thm]{Theorem}
\newtheorem{lemma}[thm]{Lemma}
\newtheorem{corollary}[thm]{Corollary}
\newtheorem{corollary*}[thm]{Corollary*}
\newtheorem{proposition}[thm]{Proposition}
\newtheorem{proposition*}[thm]{Proposition*}
\newtheorem{conjecture}[thm]{Conjecture}
\theoremstyle{definition}
\newtheorem{construction}{Construction}
\newtheorem{notations}[thm]{Notations}
\newtheorem{question}[thm]{Question}
\newtheorem{problem}[thm]{Problem}
\newtheorem{remark}[thm]{Remark}
\newtheorem{remarks}[thm]{Remarks}
\newtheorem{definition}[thm]{Definition}
\newtheorem{claim}[thm]{Claim}
\newtheorem{assumption}[thm]{Assumption}
\newtheorem{assumptions}[thm]{Assumptions}
\newtheorem{properties}[thm]{Properties}
\newtheorem{example}[thm]{Example}
\newtheorem{comments}[thm]{Comments}
\newtheorem{blank}[thm]{}
\newtheorem{observation}[thm]{Observation}
\newtheorem{defn-thm}[thm]{Definition-Theorem}

\newcommand{\sM}{{\mathcal M}}


\title{Hurwitz-Hodge integral identities from the cut-and-join equation}
      \author{Wei Luo, Shengmao Zhu}
        \address{Center of Mathematical Sciences, Zhejiang University, Hangzhou, Zhejiang 310027, China}
       \email{luowei1428@gmail.com, zhushengmao@gmail.com}


\begin{abstract}
In this paper, we present some Hurwitz-Hodge integral identities
which are derived from the Laplace transform of the cut-and-join
equation for the orbifold Hurwitz numbers. As an application, we
prove a conjecture on Hurwitz-Hodge integral proposed by J. Zhou in
2008.

\end{abstract}
\maketitle

\section{Introduction}
The Gromov-Witten theory for symplectic orbifolds has been developed
by Chen-Ruan \cite{CR}. The algebraic part, the Gromov-Witten theory
for smooth DM stacks, was established by Abramovich-Graber-Vistoli
\cite{AGV}. By the virtual localization, all the orbifold
Gromov-Witten invariants for a toric DM stack descend to the
computation of the Hurwitz-Hodge integrals on the moduli space
$\overline{\mathcal{M}}_{g,\gamma}(BG)$ of twisted stable maps to
the classifying stack of a finite group $G$ \cite{Liu}. In
\cite{Zhou1}, J. Zhou described an effective algorithm to calculate
the Hurwitz-Hodge integrals. By applying Tseng's orbifold quantum
Riemann-Roch theorem \cite{Tseng}, Hurwitz-Hodge integrals can be
reconstructed from the descendant Hurwitz-Hodge integrals on
$\overline{\mathcal{M}}_{g,\gamma}(BG)$.

Recall that the descendant Hodge integrals on moduli space of curves
can be computed by the DVV recursion which is equivalent to the
Witten-Kontsevich theorem \cite{Witten,Ko}. An easy approach to DVV
recursion is by studying the cut-and-join equation for simple
Hurwitz numbers. Combining the famous ELSV formula,  one can obtain
a polynomial identity for linear Hodge integrals. Then it is direct
to derive the DVV recursion by looking at the highest terms in this
identity \cite{CLL,MZ}. Similarly, collecting the lowest terms, one
obtains a formula for $\lambda_g$-integrals  \cite{GJV,Zhu1}.

In \cite{JPT}, Johnson-Pandharipande-Tseng established the ELSV-type
formula for orbifold Hurwitz numbers which also satisfy the
cut-and-join equation. With the same technique used in \cite{EMS},
Bouchard-Serrano-Liu-Mulase \cite{BSLM} established the Laplace
transform of the cut-and-join equation for orbifold Hurwtiz numbers.
Starting from this formula, with the same method used in
\cite{GJV,CLL,Zhu2,MZ}, we obtain the following Hurwtiz-Hodge
integrals identities.

\begin{theorem} When $r\geq 1$ and $0\leq k_i\leq r-1$, for $1\leq i \leq
l$, the Hurwitz-Hodge integrals satisfies the following orbifold-DVV
recursion:
\begin{align}
&\langle\tau_{b_{L}}\rangle_g^{r,k_L}=r\sum_{j=2}^l\frac{C^{k_1,k_j}_{b_1,b_j}}{(2b_1+1)!!(2b_j-1)!!}\langle\tau_{b_1+b_j-1}\tau_{b_{L\setminus\{1,j\}}}
\rangle_g^{r,(r\langle \frac{k_1+k_j}{r}\rangle
,k_{L\setminus\{1,j\}})}\\\nonumber
&+\frac{r^2}{2}\sum_{\substack{a+b=k_1\\m+n=b_1-2}}\left(
\frac{(2m+1)!!(2n+1)!!}{(2b_1+1)!!}\langle\tau_m\tau_n\tau_{b_{L\setminus\{1\}}}\rangle_{g-1}^{r,(a,b,k_{L\setminus\{1\}})}\right.\\\nonumber
&\left.+\sum_{\substack{g_1+g_2=g\\I\coprod
J=L\setminus\{1\}}}^{stable}\sum_{\substack{a+|k_I|\equiv0\\b+|k_J|\equiv
0}}\frac{(2m+1)!!(2n+1)!!}{(2b_1+1)!!}\langle\tau_m\tau_{b_I}
\rangle_{g_1}^{r,(a,k_I)}
\langle\tau_n\tau_{b_J}\rangle_{g_2}^{r,(b,k_J)} \right).
\end{align}
\end{theorem}
Where the constant $C^{k_1,k_j}_{b_1,b_j}$ can be calculated from
the formula (37) in Lemma 3.2.  We have not written down an explicit
formula for it in general. However, when $r=1$, then $k_i=0$, for
$i=1,..,l$.  It is easy to compute
\begin{align}
C^{0,0}_{b_1,b_j}=(2b_1+2b_j-1)!!.
\end{align}
Therefore, we have
\begin{corollary}
When $r=1$, the orbifold-DVV recursion is reduced to the ordinary
DVV recursion which is equivalent to the original Witten conjecture.
\begin{align}
&\langle\tau_{b_{L}}\rangle_g=\sum_{j=2}^l\frac{(2b_1+2b_j-1)!!}{(2b_1+1)!!(2b_j-1)!!}\langle\tau_{b_1+b_j-1}\tau_{b_{L\setminus\{1,j\}}}
\rangle_g\\\nonumber &+\frac{1}{2}\sum_{\substack{m+n=b_1-2\\m\geq
0,n\geq 0}}\left(
\frac{(2m+1)!!(2n+1)!!}{(2b_1+1)!!}\langle\tau_m\tau_n\tau_{b_{L\setminus\{1\}}}\rangle_{g-1}\right.\\\nonumber
&\left.+\sum_{\substack{g_1+g_2=g\\I\coprod
J=L\setminus\{1\}}}^{stable}\frac{(2m+1)!!(2n+1)!!}{(2b_1+1)!!}\langle\tau_m\tau_{b_I}
\rangle_{g_1} \langle\tau_n\tau_{b_J}\rangle_{g_2} \right).
\end{align}
\end{corollary}

The rank of the Hodge bundle $\mathbb{E}^U$ (see Section 2 for the
detail definition) on
$\overline{\mathcal{M}}_{g,\gamma}(B\mathbb{Z}_r)$ depends on the
monodories at the marking points. When all the monodories are
trivial, $rk \mathbb{E}^U=g$. Otherwise, the rank is given by
formula (19). By looking at the lowest terms in the formula (36), we
obtain the following theorem.

\begin{theorem}
When all the monodromies at the marking points are trivial, i.e.
$k_i=0$, for all $1\leq i\leq l$. If $b_i\geq 0$ for $1\leq i\leq
l$, and $\sum_{i=1}^lb_i=2g-3+l$, we have the following closed
formula for $\lambda_g^{U}$-integrals:
\begin{align}
\langle\tau_{b_L}\lambda_g^{U}\rangle_g^{r}=\binom{2g-3+l}{b_1,...,b_l}\langle\tau_{2g-2}\lambda_{g}^{U}\rangle_g^{r}
\end{align}
where the one-point integral
$\langle\tau_{2g-2}\lambda_{g}^{U}\rangle_g^{r}$ is determined by
the following formula given in \cite{JPT}(See the formula before
Section 4):
\begin{align}
\frac{1}{r}+\sum_{g>0}t^{2g}\langle\tau_{2g-2}\lambda_g^U\rangle_g^r=\frac{1}{r}\frac{t/2}{\sin(t/2)}.
\end{align}

Otherwise, the $\lambda^{U}_{rk\mathbb{E}^U}$-integral satisfies the
following identity:
\begin{align}
    &|Z| \langle
    \tau_{b_L}\lambda^U_{rk\mathbb{E}^U}\rangle_g^{r,k_L}=\left(\sum_{i\in Z}\sum_{j\in k_{L\setminus Z}}
    \frac{c_{b_i+b_j,0}^{k_j}}{c_{b_i,0}^0c_{b_j,0}^{k_j}}\langle
\tau_{b_i+b_j-1}\tau_{b_{L\setminus\{i,j\}}}\lambda_{rk
\mathbb{E}^U}^U
\rangle_g^{r,(k_j,k_{L\setminus\{i,j\}})}\right.\\\nonumber
&\left.+\sum_{i,j\in Z,i<
j}\frac{c_{b_i+b_j,0}^0}{c_{b_i,0}^0c_{b_j,0}^0}\langle\tau_{b_i+b_j-1}\tau_{b_{L\setminus\{i,j\}}}\lambda_{rk\mathbb{E}^U}^U
\rangle_g^{r,(0,k_{L\setminus\{i,j\}})} \right)\\\nonumber
&-\frac{r^2}{2}\sum_{i\in
Z}\left[\sum_{\substack{a+b=r\\m+n=b_i-2}}\langle\tau_m\tau_n\tau_{b_{L\setminus\{i\}}}\lambda^U_{rk
\mathbb{E}^U}
\rangle_{g-1}^{r,(a,b,k_{L\setminus\{i\}})}\frac{c_{m+1,0}^ac_{n+1,0}^b}{c_{b_i,0}^0}\right.\\\nonumber
&\left.+\sum_{\substack{g_1+g_2=g\\
I\coprod J=L\setminus\{i\}}}^{stable}\sum_{\substack{a+|k_{I}|\equiv
0\\b+|k_J|\equiv 0}}\langle \tau_m \tau_{b_I}\lambda^U_{rk
\mathbb{E}^U}\rangle_{g_{1}}^{r,(a,k_I)} \langle
\tau_n\tau_{b_J}\lambda^U_{rk
\mathbb{E}^U}\rangle_{g_2}^{r,(b,k_J)}\frac{c_{m+1,0}^ac_{n+1,0}^b}{c_{b_i,0}^0}\right]
\end{align}
where $Z\subset L$ is the index set defined as $Z=\{i|k_i=0, i\in
L\}$. In particular, when $Z$ is an empty set, both sides of the
identity (6) are zero. Moreover, the constants $c_{m,0}^k$ are given
in the formula (31): $c_{m,0}^k=\prod_{j=1}^m(k-jr)$.

\end{theorem}
As a direct application of the formula (4), we obtain
\begin{corollary} When $b_i\geq 0$ for $1\leq i\leq
l$, and $\sum_{i=1}^lb_i=2g-3+l$,
\begin{align}
\langle\tau_{b_L}\lambda_g^{U}\rangle_g^{r}=\frac{1}{r}\langle\tau_{b_L}\lambda_g\rangle_g
\end{align}
where the right hand side is the ordinary Hodge integral on moduli
space of curves:
\begin{align}
\langle\tau_{b_L}\lambda_g\rangle_g=\int_{\overline{\mathcal{M}}_{g,l}}\psi_{b_1}\cdots\psi_{b_l}\lambda_g.
\end{align}
\end{corollary}
\begin{proof}
The ordinary $\lambda_g$-integral  satisfies the following formula
which was firstly obtained by Faber-Pandharipande \cite{FP}
\begin{align}
\langle\tau_{b_L}\lambda_g\rangle_g=\binom{2g-3+l}{b_1,..,b_l}\langle
\tau_{2g-2}\lambda_g\rangle_g.
\end{align}
where the one-point integral $\langle \tau_{2g-2}\lambda_g\rangle_g$
is determined by
\begin{align}
1+\sum_{g>0}t^{2g}\langle\tau_{2g-2}\lambda_g\rangle_g=\frac{t/2}{\sin(t/2)}.
\end{align}
Hence, by formula (5),
\begin{align}
\langle \tau_{2g-2}\lambda_g^U\rangle_g^r=\frac{1}{r}\langle
\tau_{2g-2}\lambda_g\rangle_g.
\end{align}
Therefore, Corollary 1.4 is derived from the formula (4).
\end{proof}

\begin{remark}
The formula (7) in Corollary 1.4 was first conjectured by J. Zhou
\cite{Zhou3}. We would like to thank Prof. Kefeng Liu and Hao Xu for
pointing out it to us.
\end{remark}

\section{Preliminaries}
\subsection{Hurwitz-Hodge integrals}
Let $\overline{\mathcal{M}}_{g,\gamma}(B\mathbb{Z}_r)$ be the moduli
space of stable maps to the classifying space $B\mathbb{Z}_r$ where
$\gamma=(\gamma_1,...,\gamma_n)$ is a vector elements in
$\mathbb{Z}_r$. In particular, when $a=1$, the moduli space of
stable maps $\overline{\mathcal{M}}_{g,(0,..,0)}(B\mathbb{Z}_r)$ is
specialized to $\overline{\mathcal{M}}_{g,n}$. Let $U$ be the
irreducible $\mathbb{C}$-representation of $\mathbb{Z}_r$ given by
\begin{align}
\phi^{U}: \mathbb{Z}_r\rightarrow \mathbb{C}^*, \quad
\phi^{U}(1)=\exp(2\pi\sqrt{-1}/r).
\end{align}
We have the corresponding Hodge bundle
$\mathbb{E}^U_{g,\gamma}\rightarrow
\overline{\mathcal{M}}_{g,\gamma}(B\mathbb{Z}_r)$ and the Hodge
classes $\lambda_{i}^{U,g,\gamma}=c_i(\mathbb{E}^{U}_{g,\gamma})$.
More generally, for any irreducible representation $R$ of
$\mathbb{Z}_r$, we denote the corresponding Hodge bundle and Hodge
classes as $\mathbb{E}^R_{g,\gamma}$ and $\lambda_i^{R,g,\gamma}$
respectively. In the following, for brevity, we will also use the
notations $\mathbb{E}^R$ and $\lambda_i^{R}$ to denote the Hodge
bundle and Hodge classes without confusion.

The $i$-th cotangent line bundle $L_i$ on the moduli space of curves
has fiber $L_{i}|_{(C,p_1,...,p_n)}=T_{p_i}^*C$. The $\psi$-classes
on $\overline{\mathcal{M}}_{g,n}$ are defined by
\begin{align}
\psi_i=c_1(L_i)\in H^{2}(\overline{\mathcal{M}}_{g,n},\mathbb{Q}).
\end{align}
The $\psi$-classes on
$\overline{\mathcal{M}}_{g,\gamma}(B\mathbb{Z}_r)$ are defined by
pull-back via the morphism
\begin{align}
\epsilon:
\overline{\mathcal{M}}_{g,\gamma}(B\mathbb{Z}_r)\rightarrow
\overline{\mathcal{M}}_{g,n}
\end{align}
as
\begin{align}
\overline{\psi}_i=\epsilon^*(\psi_i)\in H^{2}(
\overline{\mathcal{M}}_{g,\gamma}(B\mathbb{Z}_r),\mathbb{Q}).
\end{align}

Hurwitz-Hodge integrals over
$\overline{\mathcal{M}}_{g,\gamma}(B\mathbb{Z}_r)$ are the top
intersection products of the classes $\{\lambda_i^{R}\}$ and
$\{\overline{\psi}_j\}_{1\leq j\leq n}$:
\begin{align}
\int_{\overline{\mathcal{M}}_{g,\gamma}(B\mathbb{Z}_r)}\overline{\psi}_1^{b_1}\cdots
\overline{\psi}_n^{b_n}(\lambda_1^{R})^{k_1}\cdots (\lambda_{rk
\mathbb{E}^R}^R)^{k_{rk \mathbb{E}^U}}
\end{align}

We let
\begin{align}
\Lambda^{R}(t)=\sum_{j\geq 0}^{rk \mathbb{E}^R}(-t)^j\lambda_j^{R}.
\end{align}
where $rk \mathbb{E}^R$ is the rank of $\mathbb{E}^{R}$ determined
by the orbifold Riemann-Roch formula.

\subsection{Orbifold ELSV formula}
In \cite{JPT}, Johnson-Pandharipande-Tseng established the following
ELSV-type formula for orbifold Hurwtiz number
$H^{r}_{g,l}(\mu_1,...,\mu_l)$.
\begin{theorem}
The orbifold Hurwitz number has an expression in terms of linear
Hurwtiz-Hodge integrals as follows:
\begin{align}
H^{r}_{g,l}(\mu)=r^{1-g+\sum_{i=1}^l\langle\frac{\mu_i}{r}\rangle}\prod_{i=1}^l\frac{\mu_i^{\lfloor\frac{\mu_i}{r}\rfloor}}
{\lfloor\frac{\mu_i}{r}\rfloor!}
\int_{\overline{\mathcal{M}}_{g,-\mu}(B\mathbb{Z}_r)}
\frac{\Lambda}{\prod_{i=1}^l(1-\mu_i\overline{\psi}_i)}.
\end{align}
where $\Lambda=\Lambda^{U}(r)=\sum_{j\geq
0}^{rk\mathbb{E}^U}(-r)^j\lambda^{U}_j$, $\mu=(\mu_1,..,\mu_l)$ is a
partition with length $l$. The floor and fractional part of a $q\in
\mathbb{Q}$ is denoted by $q=\lfloor q \rfloor+\langle q\rangle$.
\end{theorem}
The rank of the Hodge bundle $\mathbb{E}^U$ can be calculated by the
orbifold Riemann-Roch formula. When all the monodromies are trivial,
i.e. $\langle\frac{\mu_i}{r}\rangle=0$, for all $i=1,..,l$, the rank
of $\mathbb{E}^U_{g,\mu}$ is $g$, otherwise, the rank is (see
formula (3.17) in \cite{BSLM}) given by
\begin{align}
rk\mathbb{E}^U_{g,\mu}=g-1+\sum_{i=1}^l\langle
-\frac{\mu_i}{r}\rangle.
\end{align}

\subsection{The Laplace transform of the cut-and-join equation for orbifold Hurwitz numbers}
The orbifold Hurwitz number $H^{(r)}_{g,l}(\mu)$ also satisfies the
cut-and-join equation, the following formula was established in
Bouchard-Serrano-Liu-Mulase \cite{BSLM}.
\begin{theorem}(Cut-and-join equation \cite{BSLM})
\begin{align}
&sH^{(r)}_{g,l}(\mu)=\frac{1}{2}\sum_{i \neq
j}(\mu_i+\mu_j)H^{(r)}_{g,l-1}(\mu_i+\mu_j,\mu(\hat{i},\hat{j}))\\\nonumber
&=\frac{1}{2}\sum_{i=1}^{l}\sum_{\alpha+\beta=\mu_i}\alpha\beta
\left[H^{(r)}_{g-1,l-1}(\alpha,\beta,\mu(\hat{i}))+\sum_{\substack{g_1+g_2=g\\I\coprod
J=L\setminus\{i\}}}H_{g_1,|I|+1}^{(r)}(\alpha,\mu_I)H_{g_2,|J|+1}^{(r)}(\beta,\mu_J)\right]
\end{align}
where
\begin{align}
s=2g-2+l+\sum_{i=1}^{l}\frac{\mu_i}{r}
\end{align}
is the number of the simple ramification point given by the
Riemann-Hurwitz formula. The notation $\hat{i}$ indicates that the
parts $\mu_i$ is erased.
\end{theorem}
In order to describe the Laplace transform of the above cut-and-join
formula (20). We need to introduce the auxiliary functions showed in
\cite{BSLM} (see section 7.2) firstly:
 \makeatletter
\let\@@@alph\@alph
\def\@alph#1{\ifcase#1\or \or $'$\or $''$\fi}\makeatother
\begin{subnumcases}
{\xi_{-1}^{r,k}(\eta)=} \frac{1}{r}\eta^r, &$k=0$, \\\nonumber
\frac{1}{kr^{k/r}}\eta^k, &$0<k<r$.
\end{subnumcases}
\makeatletter\let\@alph\@@@alph\makeatother and for $m\geq -1$, we
let
\begin{align}
\xi_{m+1}^{r,k}(\eta)=\frac{\eta}{1-\eta^r}\frac{d}{d\eta}\xi_{m}^{r,k}(\eta).
\end{align}

For the following exposition, we also need to fix some notations.
let $L=(1,..,l)$ be an index set, we denote $b_L=(b_1,..,b_l)$ with
$b_i\geq 0$, $k_L=(k_1,..,k_l)$ with $0\leq k_i<r$,
$|k_L|=\sum_{i=1}^l k_i$ and $\tau_{b_L}=\prod_{i=1}^l\tau_{b_i}$,
$\xi_{b_L}^{r,k_L}(\eta_L)=\prod_{i=1}^{l}\xi_{b_i}^{r,k_i}(\eta_i)$.

The Hurwitz-Hodge integrals are abbreviated as
\begin{align}
\langle\tau_{b_L}\Lambda \rangle_g^{r,k_L}:=\langle
\tau_{b_1}\tau_{b_2}\cdots
\tau_{b_l}\Lambda\rangle_g^{r,k_L}=\int_{\overline{\mathcal{M}}_{g,-k_L}(B\mathbb{Z}_r)}\prod_{i=1}^l\psi_i^{b_i}\Lambda.
\end{align}

Now the Laplace transform of the cut-and-join equation can be
described as follows (see formula (7.26) in \cite{BSLM}).
\begin{align}
&\sum_{\substack{|k_L|\equiv0\\|b_L|\leq
3g-3+l}}r^{\frac{|k_L|}{r}}\langle\tau_{b_L}\Lambda
\rangle_g^{r,k_L}\left[(2g-2+l)\xi_{b_L}^{r,k_L}(\eta_{L})+\frac{1}{r}\sum_{i=1}^{l}(1-\eta_i^r)\xi_{b_i+1}^{r,k_i}(\eta_i)
\xi_{b_{L\setminus\{i\}}}^{r,k_{L\setminus\{i\}}}(\eta_{L\setminus\{i\}})\right]\\\nonumber
&=\sum_{1\leq i<j\leq
l}\sum_{\substack{a+|k_{L\setminus\{i,j\}}|\equiv 0\\m+b_{L\setminus
\{i,j\}\leq 3g-4+l}}}r^{\frac{a+|k_{L\setminus\{i,j\}}|}{r}}\langle
\tau_m \tau_{b_{L\setminus \{i,j\}}}\Lambda
\rangle_g^{r,(a,k_{L\setminus\{i,j\}})}\frac{1}{\eta_i-\eta_j}\\\nonumber
&\times \left[\frac{\eta_j\xi_{m+1}^{r,a}(\eta_i)}{1-\eta_i^r}
-\frac{\eta_i\xi_{m+1}^{r,a}(\eta_j)}{1-\eta_j^r}\right]
\xi_{b_{L\setminus\{i,j\}}}^{r,k_{L\setminus\{i,j\}}}(\eta_{L\setminus\{i,j\}})\\\nonumber
&+\frac{r}{2}\sum_{i=1}^l\left(\sum_{\substack{a+b+|k_{L\setminus\{i\}}|\equiv
0\\m+n+|b_{L\setminus\{i\}}|\leq
3g-5+l}}r^{\frac{a+b+|k_{L\setminus\{i\}}|}{r}}\langle\tau_m\tau_n\Lambda
\rangle_{g-1}^{r,(a,b,k_{L\setminus\{i\}})}\xi_{m+1}^{r,a}(\eta_i)\xi_{n+1}^{r,b}(\eta_i)
\xi_{b_{L\setminus\{i\}}}^{r,k_{L\setminus\{i\}}}(\eta_{L\setminus\{i\}})\right.\\\nonumber
&\left.+\sum_{\substack{g_1+g_2=g\\ I\coprod
J=L\setminus\{i\}}}^{stable}\sum_{\substack{a+|k_{I}|\equiv
0\\b+k_J\equiv 0\\m+|b_I|\leq 3g_1-2+|I|\\n+|b_J|\leq
3g_2-2+|J|}}r^{\frac{a+b+|k_{L\setminus\{i\}}|}{r}}\langle \tau_m
\tau_{b_I}\Lambda\rangle_{g_{1}}^{r,(a,k_I)} \langle
\tau_n\tau_{b_J}\Lambda\rangle_{g_2}^{r,(b,k_J)}\right.\\\nonumber
&\left.\times\xi_{m+1}^{r,a}(\eta_i)\xi_{n+1}^{r,b}(\eta_i)
\xi_{b_{L\setminus\{i\}}}^{r,k_{L\setminus\{i\}}}(\eta_{L\setminus\{i\}})\right).
\end{align}

\section{Proof of the main results}
In this section, we present the proofs  for the results showed in
Section 1. First, we introduce the new variable $t$ as
\begin{align}
t^r=\frac{1}{1-\eta^r}.
\end{align}
It is easy to get
\begin{align}
\frac{\eta}{1-\eta^r}\frac{d}{d\eta}=t^{r+1}(t^r-1)\frac{d}{dt}.
\end{align}
Hence, in the new variable $t$, \makeatletter
\let\@@@alph\@alph
\def\@alph#1{\ifcase#1\or \or $'$\or $''$\fi}\makeatother
\begin{subnumcases}
{\xi_{-1}^{r,k}(t)=} \frac{1}{r}\left(\frac{t^r-1}{t^r}\right),
&$k=0$,
\\\nonumber \frac{1}{kr^{k/r}}\frac{(t^r-1)^{\frac{k}{r}}}{t^k}, &$1\leq k\leq r-1$.
\end{subnumcases}
\makeatletter\let\@alph\@@@alph\makeatother and for $m\geq -1$, we
have
\begin{align}
\xi_{m+1}^{r,k}(t)=t^{r+1}(t^r-1)\frac{d}{dt}\xi_{m}^{r,k}(t).
\end{align}

\begin{lemma}
For $0\leq k\leq r-1$, the function $\xi_{m}^{r,k}(t)$ has the
following expansion form:
\begin{align}
\xi_{m}^{r,k}(t)=\frac{1}{r^{\frac{k}{r}}}t^{(m+1)r-k}(t^r-1)^{\frac{k}{r}}(c^k_{m,m}t^{mr}+c^k_{m,m-1}t^{(m-1)r}+\cdots
+c^k_{m,0}).
\end{align}
where the coefficients $c^k_{m,i}$, for $0 \leq i \leq m$, satisfy
certain recursion relations. In particularly, we have
\begin{align}
c^k_{m,m}=(2m-1)!!r^{m},\quad  c^k_{m,0}=\prod_{j=1}^m(k-jr).
\end{align}
\end{lemma}
\begin{proof}
By definition, we have
\begin{align}
\xi_1^{r,k}(t)&=t^{r+1}(t^r-1)\frac{d}{dt}\left(\frac{1}{r^{\frac{k}{r}}}t^r\left(\frac{t^r-1}{t^r}\right)^{\frac{k}{r}}\right)\\\nonumber
&=\frac{1}{r^{\frac{k}{r}}}t^{2r-k}(t^r-1)^{\frac{k}{r}}(r
t^r-(r-k)).
\end{align}
So $\xi_1^{r,k}(t)$ has the expansion form (30) showed in Lemma 3.1,
with
\begin{align}
c^k_{1,1}=r, \quad c^k_{1,0}=k-r.
\end{align}
When $m\geq 1$, one has
\begin{align}
\xi_{m+1}^{r,k}(t)&=t^{r+1}(t^r-1)\frac{d}{dt}\xi_{m}^{r,k}(t)\\\nonumber
&=\frac{1}{r^{\frac{k}{r}}}t^{r+1}(t^r-1)\frac{d}{dt}\left(t^{(m+1)r-k}(t^{r}-1)^{\frac{k}{r}}
\sum_{j=0}^m c^k_{m,j}t^{jr}\right)\\\nonumber &=
\frac{1}{r^{\frac{k}{r}}}t^{(m+2)r-k}(t^r-1)^{\frac{k}{r}}((2m+1)r
c^k_{m,m}t^{(m+1)r}+\cdots+(k-(m+1)r)c^k_{m,0}).
\end{align}
Therefore,
\begin{align}
c^k_{m+1,m+1}=(2m+1)rc^k_{m,m},\quad
c^k_{m+1,0}=(k-(m+1)r)c^k_{m,0}.
\end{align}
It is direct to obtain the formula (31) by the initial values in
formula (33).
\end{proof}

In terms of the new auxiliary functions $\xi_{m}^{r,k}(t)$,  the
formula (25) can be changed to
\begin{align}
&\sum_{\substack{|k_L|\equiv0\\|b_L|\leq
3g-3+l}}r^{\frac{|k_L|}{r}}\langle\tau_{b_L}\Lambda
\rangle_g^{r,k_L}\left[(2g-2+l)\xi_{b_L}^{r,k_L}(t_{L})+\frac{1}{r}\sum_{i=1}^{l}\frac{1}{t_i^r}\xi_{b_i+1}^{r,k_i}(t_i)
\xi_{b_{L\setminus\{i\}}}^{r,k_{L\setminus\{i\}}}(t_{L\setminus\{i\}})\right]\\\nonumber
&=\sum_{1\leq i<j\leq
l}\sum_{\substack{a+|k_{L\setminus\{i,j\}}|\equiv 0\\m+b_{L\setminus
\{i,j\}\leq 3g-4+l}}}r^{\frac{a+|k_{L\setminus\{i,j\}}|}{r}}\langle
\tau_m \tau_{b_{L\setminus \{i,j\}}}\Lambda
\rangle_g^{r,(a,k_{L\setminus\{i,j\}})}\xi_{b_{L\setminus\{i,j\}}}^{r,k_{L\setminus\{i,j\}}}(t_{L\setminus\{i,j\}})\\\nonumber
&\times
\left[\frac{t_i^{r+1}(t_j^r-1)^{\frac{1}{r}}\xi_{m+1}^{r,a}(t_i)-t_j^{r+1}(t_i^r-1)^{\frac{1}{r}}\xi_{m+1}^{r,a}(t_j)}
{t_j(t_i^r-1)^{\frac{1}{r}}-t_i(t_j^r-1)^{\frac{1}{r}}}\right]
\\\nonumber
&+\frac{r}{2}\sum_{i=1}^l\left(\sum_{\substack{a+b+|k_{L\setminus\{i\}}|\equiv
0\\m+n+|b_{L\setminus\{i\}}|\leq
3g-5+l}}r^{\frac{a+b+|k_{L\setminus\{i\}}|}{r}}\langle\tau_m\tau_n\tau_{b_{L\setminus\{i\}}}\Lambda
\rangle_{g-1}^{r,(a,b,k_{L\setminus\{i\}})}\xi_{m+1}^{r,a}(t_i)\xi_{n+1}^{r,b}(t_i)
\xi_{b_{L\setminus\{i\}}}^{r,k_{L\setminus\{i\}}}(t_{L\setminus\{i\}})\right.\\\nonumber
&\left.+\sum_{\substack{g_1+g_2=g\\
I\coprod J=L\setminus\{i\}}}^{stable}\sum_{\substack{a+|k_{I}|\equiv
0\\b+k_J\equiv 0\\m+|b_I|\leq 3g_1-2+|I|\\n+|b_J|\leq
3g_2-2+|J|}}r^{\frac{a+b+|k_{L\setminus\{i\}}|}{r}}\langle \tau_m
\tau_{b_I}\Lambda\rangle_{g_{1}}^{r,(a,k_I)} \langle
\tau_n\tau_{b_J}\Lambda\rangle_{g_2}^{r,(b,k_J)}\right.\\\nonumber
&\left.\times\xi_{m+1}^{r,a}(t_i)\xi_{n+1}^{r,b}(t_i)
\xi_{b_{L\setminus\{i\}}}^{r,k_{L\setminus\{i\}}}(t_{L\setminus\{i\}})\right).
\end{align}

\begin{lemma}
For  $0\leq a\leq r-1$, we have
\begin{align}
&\frac{t_i^{r+1}(t_j^r-1)^{\frac{1}{r}}\xi_{m+1}^{r,a}(t_i)-t_j^{r+1}(t_i^r-1)^{\frac{1}{r}}\xi_{m+1}^{r,a}(t_j)}
{t_j(t_i^r-1)^{\frac{1}{r}}-t_i(t_j^r-1)^{\frac{1}{r}}}\\\nonumber
&=\frac{1}{r^{\frac{a}{r}}}\sum_{p=0}^{m+1}c^a_{m+1,p}\left(\sum_{s=0}^{a-2}(t_i^r-1)^{\frac{a-1-s}{r}}(t_j^r-1)^{\frac{s+1}{r}}
t_i^{(m+3+p)r-(a-1-s)}t_j^{(m+3+p)r-(s+1)}\right.\\\nonumber
&-\left.\sum_{q=0}^{m+3+p}\binom{m+3+p}{q}(-1)^{q}\sum_{s=0}^{qr-a}(t_i^r-1)^{\frac{qr-s}{r}}(t_j^r-1)^{\frac{s+a}{r}}
t_i^{(m+3+p)r-(qr-s)}t_j^{(m+3+p)r-(s+a)}\right)
\end{align}
\begin{proof}
By a direct calculation, we have
\begin{align}
&\frac{t_i^{r+1}(t_j^r-1)^{\frac{1}{r}}\xi_{m+1}^{r,a}(t_i)-t_j^{r+1}(t_i^r-1)^{\frac{1}{r}}\xi_{m+1}^{r,a}(t_j)}
{t_j(t_i^r-1)^{\frac{1}{r}}-t_i(t_j^r-1)^{\frac{1}{r}}}\\\nonumber
&=\frac{1}{r^{\frac{a}{r}}}\frac{\frac{(t_i^r-1)^{\frac{1}{r}}}{t_i}\frac{(t_j^r-1)^{\frac{1}{r}}}{t_j}}
{\frac{(t_i^r-1)^{\frac{1}{r}}}{t_i}-\frac{(t_j^r-1)^{\frac{1}{r}}}{t_j}}\left[t_i^{(m+3)r-a+1}(t_i^r-1)^{\frac{a-1}{r}}
\sum_{p=0}^{m+1}c^a_{m+1,p}t_i^{pr}\right.\\\nonumber
&\left.-t_j^{(m+3)r-a+1}(t_j^r-1)^{\frac{a-1}{r}}\sum_{p=0}^{m+1}c^a_{m+1,p}t_j^{pr}\right.]\\\nonumber
&=\frac{1}{r^{\frac{a}{r}}}\frac{\eta_i\eta_j}{\eta_i-\eta_j}\left[\eta_i^{a-1}\sum_{p=0}^{m+1}c^a_{m+1,p}t_i^{m+3+p}
-\eta_j^{a-1}\sum_{p=0}^{m+1}c^a_{m+1,p}t_j^{m+3+p}\right]\\\nonumber
&=\frac{1}{r^{\frac{a}{r}}}\frac{\eta_i\eta_j}{\eta_i-\eta_j}\sum_{p=0}^{m+1}c^a_{m+1,p}
\left(\frac{\eta_i^{a-1}}{(1-\eta_i^{r})^{m+3+p}}-\frac{\eta_j^{a-1}}{(1-\eta_j^{r})^{m+3+p}}\right)\\\nonumber
&=\frac{1}{r^{\frac{a}{r}}}t_i^{(m+3+p)r}t_j^{(m+3+p)r}\frac{\eta_i\eta_j}{\eta_i-\eta_j}
\sum_{p=0}^{m+1}c^a_{m+1,p}\left(\eta_i^{a-1}(1-\eta_j^r)^{m+3+p}-\eta_j(1-\eta_i^r)^{m+3+p}\right)\\\nonumber
&=\frac{1}{r^{\frac{a}{r}}}t_i^{(m+3+p)r}t_j^{(m+3+p)r}\eta_i\eta_j\sum_{p=0}^{m+1}c^a_{m+1,p}\\\nonumber
&\times\left(\frac{\eta_i^{a-1}-\eta_{j}^{a-1}}{\eta_i-\eta_j}-\sum_{q=1}^{m+3+p}\binom{m+3+p}{q}(-1)^q\eta_i^{a-1}\eta_j^{a-1}
\left(\frac{\eta_i^{qr+1-a}-\eta_{j}^{qr+1-a}}{\eta_i-\eta_j}\right)\right)\\\nonumber
&=\frac{1}{r^{\frac{a}{r}}}t_i^{(m+3+p)r}t_j^{(m+3+p)r}\sum_{p=0}^{m+1}c^a_{m+1,p}\\\nonumber
&\times\left(\sum_{s=0}^{a-2}\eta_i^{a-1-s}\eta_j^{s+1}-\sum_{q=0}^{m+3+p}\binom{m+3+p}{q}(-1)^q\sum_{s=0}^{qr-a}\eta_i^{qr-s}\eta_j^{s+a}\right)\\\nonumber
&=\frac{1}{r^{\frac{a}{r}}}\sum_{p=0}^{m+1}c^a_{m+1,p}\left(\sum_{s=0}^{a-2}(t_i^r-1)^{\frac{a-1-s}{r}}(t_j^r-1)^{\frac{s+1}{r}}
t_i^{(m+3+p)r-(a-1-s)}t_j^{(m+3+p)r-(s+1)}\right.\\\nonumber
&-\left.\sum_{q=0}^{m+3+p}\binom{m+3+p}{q}(-1)^{q}\sum_{s=0}^{qr-a}(t_i^r-1)^{\frac{qr-s}{r}}(t_j^r-1)^{\frac{s+a}{r}}
t_i^{(m+3+p)r-(qr-s)}t_j^{(m+3+p)r-(s+a)}\right)
\end{align}
\end{proof}
\end{lemma}

Proof of the Theorem 1.1:
\begin{proof}
When $\sum_{j=1}^lb_j=3g-3+l$, we consider the coefficient of the
monomial
\begin{align}
\prod_{j=1}^{l}(t_j^r-1)^{\frac{k_j}{r}}t_1^{(2b_1+2)r-k_1}\prod_{j=
2}^lt_j^{(2b_j+1)r-k_j}.
\end{align}
in the formula (36). This coefficient in the left hand side is equal
to
\begin{align}
\frac{1}{r}c^{k_1}_{b_1+1,b_1+1}\prod_{j=2}^lc^{k_j}_{b_j,b_j}\langle
\tau_{b_{L}}\rangle_g^{r,k_L}.
\end{align}

When $m=b_1+b_j-1$ and $a\equiv k_1+k_j$, for $j=2,..,l$, by Lemma
3.2, we denote the coefficient of the term
\begin{align}
(t_1^r-1)^{\frac{k_1}{r}}(t_j^r-1)^{\frac{k_j}{r}}t_1^{(2b_1+2)r-k_1}t_j^{(2b_j+1)r-k_j}
\end{align}
in the formula (37) as $C^{k_1,k_j}_{b_1,b_j}$. Then the coefficient
in the first term of the right hand side is
\begin{align}
\sum_{j=2}^lC^{k_1,k_j}_{b_1,b_j}\prod_{i\neq
1,j}c^{k_i}_{b_i,b_i}\langle\tau_{b_1+b_j-1}\tau_{b_{L\setminus\{1,j\}}}
\rangle_g^{r,(r\langle \frac{k_1+k_j}{r}
\rangle,k_{L\setminus\{1,j\}})}
\end{align}
This coefficient in the second term of the right hand side is
\begin{align}
&\frac{r}{2}\prod_{j=2}^lc^{k_j}_{b_j,b_j}\sum_{\substack{a+b=k_1\\m+n=b_1-2}}\left(
c^a_{m+1,m+1}c^b_{n+1,n+1}\langle\tau_m\tau_n\tau_{b_{L\setminus\{1\}}}\rangle_{g-1}^{r,(a,b,k_{L\setminus\{1\}})}\right.\\\nonumber
&\left.\sum_{\substack{g_1+g_2=g\\I\coprod
J=L\setminus\{1\}}}^{stable}\sum_{\substack{a+|k_I|\equiv0\\b+|k_J|\equiv
0}}c^a_{m+1,m+1}c^b_{n+1,n+1}\langle\tau_m\tau_{b_I}
\rangle_{g_1}^{r,(a,k_I)}
\langle\tau_n\tau_{b_J}\rangle_{g_2}^{r,(b,k_J)} \right)
\end{align}
Where the coefficient $c^{k_i}_{b_i,b_i}$ are given by formula (31).
Collecting the formulas (40),(42) and (43) together, we obtain the
Theorem 1.1.
\end{proof}

Proof of the Theorem 1.3:
\begin{proof}
Now we consider the lowest nonzero terms in the formula (36). The
rank of the Hodge bundle $\mathbb{E}_{g,k_L}^U$ is $g$ if all the
loop monodories are trivial, i.e. $k_i=0$, for $1\leq i\leq l$.
Otherwise, the rank is
\begin{equation}
    rk \mathbb{E}_{g,k_L}^U=g-1+\sum_{i=1}^l\langle -\frac{k_i}{r}\rangle.
    \label{}
\end{equation}

First, we consider the nontrivial case. So the smallest possible
$|b_L|$ is
\begin{align}
    |b_L|&=3g-3+l-(g-1+\sum_{i=1}^l\langle
    -\frac{k_i}{r}\rangle)\\\nonumber
    &=2g-2+N(\{k_L\})+\frac{|k_L|}{r}.
    \label{}
\end{align}
where we have used the notation $N(\{k_L\})$ to denote the number of
the set $\{i| k_i=0, 1\leq i\leq l\}$.

The lowest term of a $\xi_m^{r,k}(t)$ is
\begin{equation}
    \frac{1}{r^{k/r}} c^k_{m,0} t^{(m+1)r-k} (t^r-1)^{k/r}
    \label{}
\end{equation}
where $c^k_{m,0}= \prod_{j=1}^m(k-jr)$. When $|b_L|=\sum_{j=1}^l b_j
= 2g-2 +N(\{k_L\})+\frac{|k_L|}{r}$, consider the coefficient of the
monomial
\begin{equation}
    \prod_{j=1}^l (t_j^r-1)^{k_j/r} t_j^{(b_j+1)r-k_j}
    \label{}
\end{equation}in the formula (36).

This coefficient in the left hand side of (36) is
\begin{align}
   &(-r)^{rk \mathbb{E}^U_{g,k_L}}\langle \tau_{b_L}\lambda_{rk \mathbb{E}^U} \rangle_{g}^{r,k_L} \prod_{j=1}^l c^{k_j}_{b_j,0}(2g-2+l+\frac{1}{r}\sum_{i=1}^l
   (k_i-(b_i+1)r))\\\nonumber
   &=(-r)^{rk \mathbb{E}^U_{g,k_L}}\langle \tau_{b_L}\lambda_{rk \mathbb{E}^U} \rangle_{g}^{r,k_L} \prod_{j=1}^l c^{k_j}_{b_j,0}(2g-2+\frac{|k_L|}{r}-|b_L|)\\\nonumber
   &=-N(\{k_L\})(-r)^{rk \mathbb{E}^U_{g,k_L}} \prod_{j=1}^l c^{k_j}_{b_j,0}\langle \tau_{b_L}\lambda_{rk \mathbb{E}^U}^U \rangle_{g}^{r,k_L}
\end{align}

To compute coefficients in the second term of right hand side of
(36), first find that range of $m,n$ with nonzero contribution in
$\langle
\tau_m\tau_n\tau_{b_{L\setminus\{i\}}}\Lambda\rangle_{g-1}^{r,(a,b,k_{L\setminus\{i\}})}$
is
\begin{gather}
    m+n+|b_{L\backslash \{i\}}|\geq 2g-4+N(\{a,b,k_{L\setminus\{i\}}\})+\frac{|k_{L\backslash\{i\}}|+a+b}{r} \Leftrightarrow \\
    m+n \geq b_i-2+N(\{a,b\})-N(\{k_i\})+\frac{a+b-k_i}{r}
    \label{}
\end{gather}
Range of $m,n$ in
$\langle\tau_m\tau_{b_I}\Lambda\rangle_{g_1}^{r,(a,k_{I})}\langle\tau_n\tau_{b_J}\Lambda\rangle_{g_2}^{r,(b,k_J)}$
is
\begin{gather}
    m+|b_I|\geq 2g_1-2+N(\{a,k_I\})+\frac{|k_I|+a}{r} \\
    n+|b_J|\geq 2g_2-2+N(\{b,k_J\})+\frac{|k_J|+b}{r}
\end{gather}
Hence
\begin{align}
m+n\geq b_i-2+N(\{a,b\})-N(\{k_i\})+\frac{a+b-k_i}{r}.
\end{align}

Since $0\leq a,b,k_i<r$ and $\langle \frac{a+b-k_i}{r}\rangle=0$, it
is easy to obtain that
$b_i-2+N(\{a,b\})-N(\{k_i\})+\frac{a+b-k_i}{r}$ is equal to $b_i-1$
or $b_i-2$.

On the other hand, the lowest term of
$\xi_{m+1}^{r,a}(t_i)\xi_{n+1}^{r,b}(t_i)$ is
\begin{gather}
    (t_i^r-1)^{\frac{a+b}{r}} t_i^{(m+n+4)r-a-b}
\end{gather}
Only when $k_i=0$ and $a+b=r$, this term contributes a coefficient
$-1$ to the term
\begin{align}
(t_i^r-1)^{\frac{k_i}{r}}t_{i}^{(b_i+1)r-k_i}.
\end{align}
Hence, the coefficient of the monomial $\prod_{j=1}^l
(t_j^r-1)^{k_j/r} t_j^{(b_j+1)r-k_j}$ in the second term of right
hand side is
\begin{align}
&-\frac{r}{2}\sum_{i\in
Z}\left[\sum_{\substack{a+b=r\\m+n=b_i-2}}(-r)^{rk
\mathbb{E}_{g-1,(a,b,k_{L\setminus\{i\}})}^U}\langle\tau_m\tau_n\tau_{b_{L\setminus\{i\}}}\lambda^U_{rk
\mathbb{E}^U}
\rangle_{g-1}^{r,(a,b,k_{L\setminus\{i\}})}c_{m+1,0}^ac_{n+1,0}^b\prod_{j\neq
i}c_{b_j,0}^{k_j}\right.\\\nonumber
&\left.+\sum_{\substack{g_1+g_2=g\\
I\coprod J=L\setminus\{i\}}}^{stable}\sum_{\substack{a+|k_{I}|\equiv
0\\b+|k_J|\equiv 0\\m+|b_I|=
2g_1-2+N(\{k_I\})+\frac{a+|k_I|}{r}\\n+|b_J|=
2g_2-2+N(\{k_J\})+\frac{b+|k_J|}{r}}}(-r)^{rk\mathbb{E}^U_{g_1,(a,k_I)}+rk\mathbb{E}^U_{g_2,(b,k_J)}}\right.\\\nonumber
&\left.\times \langle \tau_m \tau_{b_I}\lambda^U_{rk
\mathbb{E}^U}\rangle_{g_{1}}^{r,(a,k_I)} \langle
\tau_n\tau_{b_J}\lambda^U_{rk
\mathbb{E}^U}\rangle_{g_2}^{r,(b,k_J)}c_{m+1,0}^ac_{n+1,0}^b\prod_{j\neq
i}c_{b_j,0}^{k_j}\right]
\end{align}

Similarly for the first term of right hand side, range of $m$ is
\begin{align}
m\geq N(\{a\})-N(\{k_i,k_j\})+b_i+b_j+\frac{a-k_i-k_j}{r}.
\end{align}
From the formula (37), only when $k_i, k_j=0$ and $a=0$ or $k_i=0$
and $a=k_j$, we take  $m=b_i+b_j-1$, then the left hand side of the
formula (37) can contribute a coefficient
$-\frac{1}{r^{\frac{a}{r}}}c^a_{b_i+b_j,0}$ to the term
$(t_i^r-1)^{\frac{k_i}{r}}t_i^{(b_i+1)r-k_i}(t_j^r-1)^{\frac{k_j}{r}}t_j^{(b_j+1)r-k_j}$.

If we define the index set $Z\subset L$ as
\begin{align}
Z=\{i|k_i=0, 1\leq i\leq l\}.
\end{align}
The coefficient of the monomial $\prod_{j=1}^l (t_j^r-1)^{k_j/r}
t_j^{(b_j+1)r-k_j}$ in the first term of right hand side is
\begin{align}
&-\left(\sum_{i\in Z}\sum_{j\in k_{L\setminus Z}}(-r)^{rk
\mathbb{E}^U_{g,(k_j,k_{L\setminus\{i,j\}})}}\langle
\tau_{b_i+b_j-1}\tau_{b_{L\setminus\{i,j\}}}\lambda_{rk
\mathbb{E}^U}^U
\rangle_g^{r,(k_j,k_{L\setminus\{i,j\}})}c_{b_i+b_j,0}^{k_j}\right.\\\nonumber
&\left.+\sum_{i,j\in Z,i< j}(-r)^{rk
\mathbb{E}^U_{g,(0,k_{L\setminus\{i,j\}})}}\langle\tau_{b_i+b_j-1}\tau_{b_{L\setminus\{i,j\}}}\lambda_{rk\mathbb{E}^U}^U
\rangle_g^{r,(0,k_{L\setminus\{i,j\}})}c_{b_i+b_j,0}^0
\right)\prod_{q\neq i,j}c^{k_q}_{q,0}.
\end{align}
By using the rank formula (19), it is easy to compute that the ranks
appearing in the formulas (48), (56) and (59) satisfy:
\begin{align}
rk \mathbb{E}^U_{g,(k_j,k_{L\setminus\{i,j\}})}&=rk
\mathbb{E}^U_{g,k_L}\\\nonumber rk
\mathbb{E}^U_{g,(0,k_{L\setminus\{i,j\}})}&=rk
\mathbb{E}^U_{g,k_L}\\\nonumber rk
\mathbb{E}^U_{g-1,(a,b,k_{L\setminus\{i\}})}&=rk
\mathbb{E}^U_{g,k_L}+1\\\nonumber rk \mathbb{E}^U_{g_1,(a,k_I)}+rk
\mathbb{E}^U_{g_2,(b,k_J)}&=rk \mathbb{E}^U_{g,k_L}+1.
\end{align}

 Finally, combining (48), (56) and (59)
together, we obtain the following formula
\begin{align}
    &|Z|\prod_{j=1}^l c^{k_j}_{b_j,0} \langle
    \tau_{b_L}\lambda^U_{rk\mathbb{E}^U}\rangle_g^{r,k_L}\\\nonumber
    &=\left(\sum_{i\in Z}\sum_{j\in k_{L\setminus Z}}\langle
\tau_{b_i+b_j-1}\tau_{b_{L\setminus\{i,j\}}}\lambda_{rk
\mathbb{E}^U}^U
\rangle_g^{r,(k_j,k_{L\setminus\{i,j\}})}c_{b_i+b_j,0}^{k_j}\right.\\\nonumber
&\left.+\sum_{i,j\in Z,i<
j}\langle\tau_{b_i+b_j-1}\tau_{b_{L\setminus\{i,j\}}}\lambda_{rk\mathbb{E}^U}^U
\rangle_g^{r,(0,k_{L\setminus\{i,j\}})}c_{b_i+b_j,0}^0
\right)\prod_{q\neq i,j}c^{k_q}_{q,0}\\\nonumber
&-\frac{r^2}{2}\sum_{i\in
Z}\left[\sum_{\substack{a+b=r\\m+n=b_i-2}}\langle\tau_m\tau_n\tau_{b_{L\setminus\{i\}}}\lambda^U_{rk
\mathbb{E}^U}
\rangle_{g-1}^{r,(a,b,k_{L\setminus\{i\}})}c_{m+1,0}^ac_{n+1,0}^b\prod_{j\neq
i}c_{b_j,0}^{k_j}\right.\\\nonumber
&\left.+\sum_{\substack{g_1+g_2=g\\
I\coprod J=L\setminus\{i\}}}^{stable}\sum_{\substack{a+|k_{I}|\equiv
0\\b+|k_J|\equiv 0}}\langle \tau_m \tau_{b_I}\lambda^U_{rk
\mathbb{E}^U}\rangle_{g_{1}}^{r,(a,k_I)} \langle
\tau_n\tau_{b_J}\lambda^U_{rk
\mathbb{E}^U}\rangle_{g_2}^{r,(b,k_J)}c_{m+1,0}^ac_{n+1,0}^b\prod_{j\neq
i}c_{b_j,0}^{k_j}\right]
\end{align}

So we have the formula (6) in Theorem 1.3.

As to the trivial monodromies case, i.e. $k_i=0$, for all  $1\leq
i\leq l$. We have $rk\mathbb{E}^U=g$. With the analogue analysis as
above, when $\sum_{i}b_i=2g-3+l$, the coefficients of
$\prod_{i=1}^{l}t_i^{(2b_i+1)r}$ at the left hand of formula (36)
takes the form
\begin{align}
(1-l)\langle\tau_{b_L}\lambda_g^U
\rangle_g^r\prod_{i=1}^{l}c^0_{b_i,0}.
\end{align}
While at the right hand side, only the first term has the
contribution
\begin{align}
\sum_{1\leq i<j\leq
l}\langle\tau_{b_i+b_j-1}\tau_{b_{L\setminus\{i,j\}}}\lambda_{g}^U
\rangle_g^{r}(-c^0_{b_i+b_j,0})\prod_{q\neq i,j}c^0_{b_q,0}.
\end{align}
where the coefficients $c^0_{b_i,0}$ is given by formula (31). So we
obtain
\begin{align}
\langle\tau_{b_L}
\lambda^U_{g}\rangle_g^{r}=\frac{1}{l-1}\sum_{1\leq i< j\leq
l}\frac{(b_i+b_j)!}{b_i!b_j!}\langle\tau_{b_i+b_j-1}\tau_{b_L\setminus\{i,j\}}\lambda_g^U\rangle_g^{r}.
\end{align}
Then by induction, we obtain the formula (4) in Theorem 1.3  .

\end{proof}

\proof[Acknowledgements] The authors would like to thank Hao Xu for
bringing the paper \cite{Zhou3} to their attentions. Thank Prof.
Kefeng Liu for useful discussions. This research is supported by
China Postdoctoral Science Foundation 2011M500986 and National
Science Foundation of China grants No.11201417.

$$ \ \ \ \ $$

\end{document}